\documentclass[12pt]{amsart}
\usepackage{amssymb,amscd,amsmath,amsthm,url}
\usepackage[all]{xy}
\usepackage{fullpage}
\usepackage{enumitem}
\usepackage{tikz-cd}
\usepackage{appendix}
\usepackage{xcolor}

\usepackage[colorlinks=true,
            citecolor=blue,  
            linkcolor=black, 
            urlcolor=blue]{hyperref}

\hypersetup{linktoc=all} 
\makeatletter
\AtBeginDocument{%
  \renewcommand*{\@linkcolor}{red}%
}
\makeatother

\newtheorem{Theorem}{Theorem}[section]
\newtheorem{Proposition}{Proposition}[section]
\theoremstyle{remark}
\newtheorem{Lemma}{Lemma}[section]

\newtheorem{Definition}{Definition}[section]

\newtheorem{Example}{Example}[section]
\newtheorem{Remark}{Remark}[section]
\newtheorem{Notation}{Notation}

\numberwithin{equation}{section}

\author{Ahmadreza Khazaeipoul}
\address{Department of Mathematics, University of Toronto, 40 St. George Street, Toronto, ON, Canada M5S 2E4}
\email{Ahmadreza.Khazaeipoul@mail.utoronto.ca}

\def \R {\mathbb R}

\begin{document}
\title{The symplectic groupoid for the Adler-Gelfand-Dikii Poisson structure}
\maketitle
\begin{abstract}
The Adler-Gelfand-Dikii Poisson structure arises naturally in the study of $n$-th order differential operators on the circle and plays a central role in Poisson geometry and integrable systems. 
Let $G$ be one of the Lie groups $\mathrm{PSL}(n)$, $\mathrm{PSp}(n)$ (for even $n$), or $\mathrm{PSO}(n)$ (for odd $n$). 
In this paper, we construct the symplectic groupoid integrating the Adler-Gelfand-Dikii Poisson structure associated to $G$ and prove that it is Morita equivalent to the quasi-symplectic groupoid integrating the Dirac structure on $Y_n(\mathbf{C})$, where $Y_n(\mathbf{C})$ denotes the quotient of the space of quasi-periodic non-degenerate curves by homotopies preserving the monodromy.

\end{abstract}
\tableofcontents

\section{Introduction}
The \emph{Adler-Gelfand-Dikii Poisson structure}, introduced in the seminal works of Adler, Gelfand, and Dikii~\cite{adler1978trace,gelfand1978family,gelfand1975asymptotic}, is a Poisson structure on the infinite-dimensional space of \(n\)th-order differential operators on the circle. When \(n=2\) and the Lie group is \(\mathrm{PSL}(2,\mathbb{R})\), this space identifies with the dual of the Virasoro algebra at level one, denoted \(\mathfrak{vir}_1^*(S^1)\).

Let \(\mathbf{C}\) be a compact, connected, oriented 1-manifold (diffeomorphic to \(S^1\)). For each of the Lie groups
\[
G \in \{\mathrm{PSL}(n),\ \mathrm{PSp}(n)\ \text{(for even } n),\ \mathrm{PSO}(n)\ \text{(for odd } n)\},
\]
let \(\mathcal{R}_n^G(\mathbf{C})\) denote the space of \(n\)-th order differential operators acting between sections of the \(\tfrac{1-n}{2}\)- and \(\tfrac{1+n}{2}\)-density bundles on \(\mathbf{C}\), having principal symbol equal to one, and satisfying the following additional conditions depending on \(G\):
\begin{itemize}
    \item for \(G = \mathrm{PSL}(n)\), the sub-principal symbol vanishes;
    \item for \(G = \mathrm{PSp}(n)\) (even \(n\)), the operators are self-adjoint;
    \item for \(G = \mathrm{PSO}(n)\) (odd \(n\)), the operators are skew-adjoint.
\end{itemize}

The Drinfeld-Sokolov reduction \cite[Sec.~8]{drinfeld1985lie} realizes \(\mathcal{R}_n^G(\mathbf{C})\) as a Marsden-Weinstein reduced space. The symplectic leaves of this Poisson structure were classified by Khesin and Ovsienko \cite[Thm.~4]{ovsienko1990symplectic}. In recent work, Alekseev and Meinrenken~\cite{alekseev2024coadjoint} refined the Goldman-Segal correspondence~\cite{goldman1980discontinuous,segal1981unitary} between coadjoint orbits of $\mathfrak{vir}_1^*(\mathbf{C})$ and conjugacy classes in an open subset $\widetilde{\mathrm{SL}}(2,\mathbb{R})_+ \subset \widetilde{\mathrm{SL}}(2,\mathbb{R})$ by lifting it to a Morita equivalence between the symplectic groupoid integrating the Poisson structure on $\mathfrak{vir}_1^*(\mathbf{C})$ and a quasi-symplectic groupoid integrating the Cartan-Dirac structure on $\widetilde{\mathrm{SL}}(2,\mathbb{R})_+$.

We generalize their result by constructing a symplectic groupoid integrating the Adler-Gelfand-Dikii Poisson structure on \(\mathcal{R}_n^G(\mathbf{C})\), and by proving that it is Morita equivalent to the quasi-symplectic groupoid integrating the Dirac structure on \(Y_n(\mathbf{C})\), the quotient of \(\mathcal{D}_n(\mathbf{C})\) by homotopies of non-degenerate curves preserving the monodromy.

\medskip
The paper is organized as follows: In Section~2, we study the relationship between quasi-periodic non-degenerate curves and differential operators in $\mathcal{R}_n^G(\mathbf{C})$. In Section~3, we construct the symplectic groupoid integrating the Adler-Gelfand-Dikii Poisson structure and establish a Morita equivalence with the quasi-symplectic groupoid on $Y_n(\mathbf{C})$.

 We define
\[
\mathcal{D}_n^G(\mathbf{C}) = \{\, \tilde{\mathbf{C}} \to \R\mathbf{P}^{\,n-1} \mid \text{solution curves of operators in } \mathcal{R}_n^G(\mathbf{C}) \,\}.
\]
In Section~3, we prove the main result:
\begin{Theorem}\label{Theorem 1.1}
Let \(G\) be one of the Lie groups listed above. Then the following Morita equivalence of quasi-symplectic groupoids holds:
\[
\begin{tikzcd}[column sep=large, row sep=large]
 {(\mathcal{D}_n^G(\mathbf{C}) \times_{Y_n(\mathbf{C})} \mathcal{D}_n^G(\mathbf{C})/G,\, \tilde{\omega}_1)} 
 \arrow[d, shift right] \arrow[d, shift left] 
 & {(\mathcal{D}_n^G(\mathbf{C}),\, \varpi_{\mathcal{D}})} 
     \arrow[ld, "\mathcal{D}_p"'] 
     \arrow[rd, "\widetilde{\mathcal{D}_q}"] 
 & {(G \ltimes Y_n(\mathbf{C}),\, \tilde{\omega}_2)} 
     \arrow[d, shift right] \arrow[d, shift left] \\
\mathcal{R}_n^G(\mathbf{C}) 
 & 
 & 
 Y_n(\mathbf{C})
\end{tikzcd}
\]
\end{Theorem}

Here \(\varpi_{\mathcal{D}} \in \Omega^2(\mathcal{D}_n^G(\mathbf{C}))\) denotes the 2-form described in Theorem~\ref{thm:varpiD}. This Morita equivalence induces a one-to-one correspondence between the associated Hamiltonian spaces. The proof of Theorem~\ref{Theorem 1.1} is based on a reduction of the following Morita equivalence, originally due to Alekseev-Meinrenken~\cite{alekseev2024coadjoint}:
\[
\begin{tikzcd}[column sep=large, row sep=large]
{(\mathrm{Gau}_0(Q) \ltimes \mathcal{A}(Q),\, \omega_1)} 
    \arrow[d, shift right] 
    \arrow[d, shift left] 
& {(\mathcal{P}(Q),\, \varpi_{\mathcal{P}})} 
    \arrow[ld, "\mathcal{P}_p"'] 
    \arrow[rd, "\widetilde{\mathcal{P}_q}"] 
& {(G \ltimes \widetilde{G},\, \omega_2)} 
    \arrow[d, shift right] 
    \arrow[d, shift left] \\
\mathcal{A}(Q) 
& 
& 
\widetilde{G}
\end{tikzcd}
\]
To carry this out, we prove the following proposition in the appendix.
\begin{Proposition}
Suppose the Lie 2-group \(H\times H \rightrightarrows H\) has a Hamiltonian action on a symplectic groupoid \((\mathcal{G}, \omega) \rightrightarrows (M, \pi)\). Let \(x \in \mathrm{Lie}(H)^*\) be a regular value of the moment map \(\mu_M\), and let \((x,x) \in \mathrm{Lie}(H)^*\times \mathrm{Lie}(H)^*\) be a regular value of the moment map \(\mu_{\mathcal{G}}\). Assume the stabilizers of \(x\) and \((x,x)\) are \(H\) and \(H\times H\), respectively. If the actions of \(H\times H\) on \(\mu_{\mathcal{G}}^{-1}(x,x)\) and of \(H\) on \(\mu_M^{-1}(x)\) are both free and proper, then the quotient
\[
(\mathcal{G}//(H\times H), \omega_{\mathrm{red}}) \rightrightarrows (M//H, \pi_{\mathrm{red}})
\]
is a symplectic groupoid.
\end{Proposition}
\subsection*{Acknowledgements} 
I would like to thank my advisor, Eckhard Meinrenken, for introducing me to this problem, for his constant support, and for many insightful discussions.
\section{Curves and Differential Operators}
In this section, we explore the connection between operators in $\mathcal{R}_n(\bold{C})$ and quasi-periodic non-degenerate curves in $\mathbb{RP}^{\,n-1}$. We follow the notations and conventions from \cite{alekseev2024coadjoint,ovsienko2004projective}.

\subsection{Differential operators}
\begin{Notation}
    Throughout the paper, $\mathbf{C}$ denotes a compact, connected, oriented $1$-manifold. Let $\tilde{\mathbf{C}}$ be a connected covering space of $\mathbf{C}$ such that $\mathbf{C}=\tilde{\mathbf{C}}/\mathbb{Z}$. The quotient map is denoted by $\pi: \tilde{\mathbf{C}}\to \mathbf{C}$, and the deck transformation of $\tilde{\mathbf{C}}$ corresponding to translation by $1\in\mathbb{Z}$ is denoted by $\kappa$.
\end{Notation}

In this paper, we work with real densities. For any $r\in\mathbb{R}$, we denote the \emph{$r$-density bundle} over $\mathbf{C}$ by $|\Lambda|^r_{\mathbf{C}}\to \mathbf{C}$. Its space of sections is
\[
|\Omega|^r_{\mathbf{C}}=\Gamma(|\Lambda|^r_{\mathbf{C}}).
\]
The orientation of $\mathbf{C}$ identifies $|\Omega|^r_{\mathbf{C}}$ with functions for $r=0$, $1$-forms for $r=1$, and vector fields for $r=-1$. If
\[
L:|\Omega|_{\mathbf{C}}^{r_1}\longrightarrow |\Omega|_{\mathbf{C}}^{r_2}
\]
is a differential operator of order $n$, its principal symbol is an element
\[
\sigma_n(L)\in \Gamma\,\big(\operatorname{Sym}^n(T\mathbf{C})\otimes \operatorname{Hom}(|\Lambda|_{\mathbf{C}}^{r_1}, |\Lambda|_{\mathbf{C}}^{r_2})\big)\cong |\Omega|_{\mathbf{C}}^{\,r_2-r_1-n}.
\]
Hence $\sigma_n(L)$ is scalar-valued precisely when $r_2-r_1=n$. The principal symbol of a composition is the product of principal symbols.

The \emph{formal adjoint} of $L$ is the order-$n$ operator
\[
L^*:|\Omega|_{\mathbf{C}}^{1-r_2}\to|\Omega|_{\mathbf{C}}^{1-r_1},\qquad
\int_{\mathbf{C}}(Lu)\,v=\int_{\mathbf{C}}u\,(L^*v).
\]
Its principal symbol satisfies
\[
\sigma_n(L^*)=(-1)^n\sigma_n(L).
\]
If $r_2+r_1=1$, then $L$ and $L^*$ act on the same density bundles. Thus, to have both a scalar principal symbol and $L,L^*$ acting on the same bundles, we choose $r_1=\tfrac{1-n}{2}$ and $r_2=\tfrac{1+n}{2}$. In this case, we call $L$ \emph{self-adjoint} if $L^*=L$ and \emph{skew-adjoint} if $L^*=-L$.

\begin{Remark}
Let $\mathbf{C}=S^1$ and
\[
L: |\Omega|^{\frac{1-n}{2}}_{\mathbf{C}} \longrightarrow |\Omega|^{\frac{1+n}{2}}_{\mathbf{C}},\qquad
L=\sum_{i=0}^{n} a_i(\theta)\,\frac{d^i}{d\theta^i},
\]
with $\theta$ the coordinate on $S^1$. For $\mu=u(\theta)\,|\partial\theta|^{\frac{n-1}{2}}$, one has
\begin{equation}\label{2.1}
L^*(\mu)=\Big(\sum_{i=0}^{n} (-1)^i \frac{d^i}{d\theta^i}\big(a_i(\theta)\,u(\theta)\big)\Big)\,|\partial\theta|^{\frac{n+1}{2}}.
\end{equation}
\end{Remark}

\begin{Example}
Every $v\in \mathrm{Vect}(\mathbf{C})$ defines a first-order operator $\mathcal{L}_v:|\Omega|^{r_1}_{\mathbf{C}}\to|\Omega|^{r_1}_{\mathbf{C}}$ (Lie derivative) with $\sigma_1(\mathcal{L}_v)=v$. Its formal adjoint, as an operator $|\Omega|^{1-r_1}_{\mathbf{C}}\to|\Omega|^{1-r_1}_{\mathbf{C}}$, is $\mathcal{L}_v^*=-\mathcal{L}_v$.
\end{Example}

There exists an $n$-linear, skew-symmetric \emph{multidifferential} operator, the \emph{Wronskian},
\[
W: |\Omega|^{r}_{\tilde{\mathbf{C}}} \otimes \cdots \otimes |\Omega|^{r}_{\tilde{\mathbf{C}}} \longrightarrow |\Omega|^{\,n\big(r+\frac{n-1}{2}\big)}_{\tilde{\mathbf{C}}},
\]
given in a local coordinate $\theta$ by
\[
W[u_1,\ldots,u_n]=
\det\!\begin{pmatrix}
u_1 & \cdots & u_n \\
\vdots & & \vdots \\
u_1^{(n-1)} & \cdots & u_n^{(n-1)}
\end{pmatrix}.
\]
We are particularly interested in the case $r=\tfrac{1-n}{2}$, where the Wronskian is a scalar. The Wronskian of solutions $u_1,\ldots,u_n$ of a differential operator $L$ is non-zero if and only if these solutions form a fundamental system. If $L$ has principal symbol $1$, the Wronskian satisfies the Liouville equation (see \cite[Sec.~3]{khesin1992nondegenerate})
\[
\frac{d}{d\theta}W[u_1,\ldots,u_n]=a_{n-1}(\theta)\,W[u_1,\ldots,u_n].
\]

The \emph{subprincipal symbol} of the $n$-th order operator
\begin{equation}\label{2.2}
L: |\Omega|^{\frac{1-n}{2}}_{\mathbf{C}} \longrightarrow |\Omega|^{\frac{1+n}{2}}_{\mathbf{C}}
\end{equation}
is defined by
\[
\sigma^{\mathrm{sub}}(L):=\sigma_{n-1}\!\left(\frac{L-(-1)^n L^*}{2}\right)\in |\Omega|^1_{\mathbf{C}}.
\]
In these terms, the Liouville equation becomes
\[
d\,W[u_1,\ldots,u_n]=\sigma^{\mathrm{sub}}(L)\,W[u_1,\ldots,u_n].
\]
In particular, if $\sigma^{\mathrm{sub}}(L)=0$, then $W$ is constant; hence one can choose a fundamental system with $W\equiv 1$.
Consider the $n$-th order operator
\[
L: |\Omega|^{\tfrac{1-n}{2}}_{\mathbf{C}} \longrightarrow |\Omega|^{\tfrac{1+n}{2}}_{\mathbf{C}},
\]
whose principal symbol is $1$. Let $u=(u_1,\ldots,u_n)$ be a fundamental system, $u_i\in |\Omega|^{\frac{1-n}{2}}_{\tilde{\mathbf{C}}}$. Then
\[
\kappa^*u_i=\sum_j M_{ij}\,u_j.
\]
The matrix $M_u = (M_{ij})$ is called the \emph{monodromy matrix}. If $v$ denotes another fundamental system of solutions, then the corresponding monodromy matrix $M_v$ is conjugate to $M_u$. Hence, the conjugacy class is determined by $L$.
\begin{Lemma}\label{Lemma 2.2}
If $\sigma^{\mathrm{sub}}(L)=0$ for \eqref{2.2}, then the monodromy matrix of $L$ lies in $\mathrm{SL}(n,\mathbb{R})$.
\end{Lemma}

\begin{proof}
When $\sigma^{\mathrm{sub}}(L)=0$, the Wronskian is constant and gives a parallel volume form on the solution space, which the monodromy preserves.
\end{proof}

\begin{Lemma}\label{Lemma 2.3}
If $n$ is even and \eqref{2.2} is self-adjoint, then its monodromy lies in $\mathrm{Sp}(n,\mathbb{R})$.
\end{Lemma}
\begin{proof}
In this case, the solution space carries a natural symplectic form which is preserved by parallel transport; hence the monodromy is symplectic. See \cite[Sec.~2.2]{beilinson2005opers}.
\end{proof}

\begin{Lemma}\label{Lemma 2.4}
If $n$ is odd and \eqref{2.2} is skew-adjoint, then its monodromy lies in $\mathrm{SO}(n)$.
\end{Lemma}
\begin{proof}
The solution space carries a natural non-degenerate symmetric bilinear form, preserved by parallel transport, so the monodromy is orthogonal. See \cite[Sec.~2.2]{beilinson2005opers}.
\end{proof}

\begin{Notation}
If $E\to \mathbf{C}$ is a vector bundle, its $r$-th \emph{jet bundle} is $J^r(E)\to \mathbf{C}$.
\end{Notation}

\begin{Notation}
For $k\ge 0$ and vector bundles $E,F$ over $\mathbf{C}$, $\operatorname{DO}^{k}(E,F)$ denotes the space of $k$-th order differential operators $\Gamma(E)\to \Gamma(F)$.
\end{Notation}
Let $\mathcal{R}^G_n(\mathbf{C})$ be the space of $n$-th order operator $\eqref{2.2}$ with principal symbol $1$, satisfying:
\begin{enumerate}
    \item $G=\mathrm{PSL}(n)$: $\sigma^{\mathrm{sub}}(L)=0$;
    \item $G=\mathrm{PSp}(n)$ with $n$ even: $L$ is self-adjoint;
    \item $G=\mathrm{PSO}(n)$ with $n$ odd: $L$ is skew-adjoint.
\end{enumerate}

\begin{Notation}
For $G=\mathrm{PSL}(n)$, write $\mathcal{R}_n(\mathbf{C})$ for $\mathcal{R}^G_n(\mathbf{C})$.
\end{Notation}

Then $\mathcal{R}_n(\mathbf{C})$ is an affine space modeled on
\[
\operatorname{DO}^{\,n-2}\!\big(|\Lambda|_{\mathbf{C}}^{\frac{1-n}{2}},\,|\Lambda|_{\mathbf{C}}^{\frac{1+n}{2}}\big).
\]
For even $n$ ($G=\mathrm{PSp}(n)$) and odd $n$ ($G=\mathrm{PSO}(n)$), the spaces $\mathcal{R}^G_n(\mathbf{C})$ are the intersections of the self-adjoint and skew-adjoint loci, respectively, with $\mathcal{R}_n(\mathbf{C})$; their underlying linear spaces are the corresponding self-/skew-adjoint subspaces of $\operatorname{DO}^{\,n-2}$ above.

\subsection{Solution curves}
\begin{Notation}
Throughout this subsection, $G\in\{\mathrm{PSL}(n),\mathrm{PSp}(n)\text{ (even $n$)},\mathrm{PSO}(n)\text{ (odd $n$)}\}$.
\end{Notation}

\begin{Definition}
A curve
\[
\gamma:\tilde{\mathbf{C}}\longrightarrow \mathbb{RP}^{\,n-1}
\]
is \emph{quasi-periodic non-degenerate} with monodromy $h\in G$ if:
\begin{enumerate}
    \item \label{cond:nondegenerate} Let $F_k(x)$ be the projective span of $\gamma'(x),\ldots,\gamma^{(k)}(x)$. Then
    \[
    F_1(x)\subset \cdots \subset F_{n-1}(x)=\mathbb{RP}^{\,n-1}.
    \]
    \item There exists $h\in G$ such that
    \[
    \gamma(\kappa(x))=h\cdot \gamma(x).
    \]
\end{enumerate}
\end{Definition}

Condition~\ref{cond:nondegenerate} is equivalent to: for any lift $\Gamma$ of $\gamma$ to $\mathbb{R}^n$ and every $x\in \tilde{\mathbf{C}}$,
\[
\det\!\big(\Gamma(x),\Gamma'(x),\ldots,\Gamma^{(n-1)}(x)\big)\neq 0.
\]

For each $x\in \tilde{\mathbf{C}}$, the subspace $F_{n-2}(x)\subset \mathbb{RP}^{\,n-1}$ is a hyperplane in $\mathbb{R}^n$, which determines a point of the dual projective space $\mathbb{RP}^{\,n-1}{}^* \cong \mathbb{RP}^{\,n-1}$. Hence we define the \emph{dual curve}
\[
\gamma^*:\tilde{\mathbf{C}}\to \mathbb{RP}^{\,n-1},\qquad \gamma^*(x):=\text{ann}(F_{n-2}(x)).
\]

\begin{Lemma}[\cite{ovsienko2004projective}, Thm.~1.1.3]
If $\gamma$ is quasi-periodic non-degenerate, then $\gamma^*$ is quasi-periodic non-degenerate and $(\gamma^*)^*=\gamma$.
\end{Lemma}

We call $\gamma$ \emph{self-dual} if $\gamma^*=\gamma$.

Suppose $L\in \mathcal{R}_n(\mathbf{C})$ and $u_1,\ldots,u_n\in |\Omega|^{\frac{1-n}{2}}_{\tilde{\mathbf{C}}}$ is a fundamental system of solutions. Following \cite[Def.~3.2]{ovsienko1990symplectic}, the curve
\[
\gamma:\tilde{\mathbf{C}}\to \mathbb{RP}^{\,n-1},\qquad x\mapsto (u_1(x):\cdots:u_n(x)),
\]
is quasi-periodic non-degenerate with monodromy in $\mathrm{PSL}(n)$. If $L\in \mathcal{R}^G_n(\mathbf{C})$, then by Lemmas \ref{Lemma 2.2}–\ref{Lemma 2.4}, the monodromy of $\gamma$ lies in $G$. For these Lie groups, the Wronskian $W[u_1,\ldots,u_n]$ is constant.

\subsection{From curves to differential operators}
Let $(u_1,\ldots,u_n)$ be a lift of a quasi-periodic non-degenerate curve $\gamma$ to $\mathbb{R}^n$ with Wronskian $1$. Define
\begin{equation}\label{2.3}
L(u)=(-1)^nW[u,u_1,\ldots,u_n].
\end{equation}
The definition \eqref{2.3} does not depend on the choice of lift of $\gamma$ to $\mathbb{R}^n$.
\begin{Definition}
Let $\mathcal{D}_n^G(\mathbf{C})$ be the space of quasi-periodic non-degenerate curves
\[
\gamma:\tilde{\mathbf{C}} \to \mathbb{RP}^{\,n-1}
\]
with:
\begin{enumerate}
    \item $G=\mathrm{PSL}(n)$: $\gamma$ admits a lift with constant Wronskian;
    \item $G=\mathrm{PSp}(n)$, $n$ even: $\gamma$ is self-dual;
    \item $G=\mathrm{PSO}(n)$, $n$ odd: $\gamma$ is self-dual.
\end{enumerate}
\end{Definition}

\begin{Notation}
Write $\mathcal{D}_n(\mathbf{C})$ for $\mathcal{D}_n^G(\mathbf{C})$ when $G=\mathrm{PSL}(n)$.
\end{Notation}

\begin{Proposition}[\cite{ovsienko2004projective}, Thm.~2.2.6]
Let $L: |\Omega|_{\mathbf{C}}^{\frac{1-n}{2}} \to |\Omega|_{\mathbf{C}}^{\frac{1+n}{2}}$ be the operator with principal symbol $1$ associated with the quasi-periodic non-degenerate curve $\gamma:\tilde{\mathbf{C}}\to \mathbb{RP}^{\,n-1}$. Then the operator associated with the dual curve $\gamma^*$ is $(-1)^n L^*$.
\end{Proposition}
Therefore, elements of $\mathcal{D}_n^G(\mathbf{C})$ correspond to operators in $\mathcal{R}_n^G(\mathbf{C})$, and conversely. This defines a principal $G$-bundle
\[
\mathcal{D}_p:\mathcal{D}_n^G(\mathbf{C})\to \mathcal{R}_n^G(\mathbf{C}),
\] 
where the $G$-action on $\gamma\in \mathcal{D}_n^G(\mathbf{C})$ is
\[
(g\cdot \gamma)(x)=g\cdot \gamma(x),\qquad g\in G,\; x\in \tilde{\mathbf{C}}.
\]
Associating to $\gamma$ its monodromy yields
\[
\mathcal{D}_q:\mathcal{D}_n^G(\mathbf{C})\to G,\qquad \gamma\mapsto h.
\]
Thus, the ambiguity of monodromy up to conjugacy for differential equations on a circle is resolved by passing to solution curves.
Note that $\mathcal{D}_q$ is $G$-equivariant for the conjugation action on $G$.

\begin{Remark}
For $G=\mathrm{SL}(3,\mathbb{R})$ and $\bold{C}=S^1$, taking third-order linear ODEs with constant coefficients shows that the conjugacy classes of the following Jordan forms with determinant $1$ lie in the image of $\mathcal{D}_q$:
\[
\begin{pmatrix}\lambda&0&0\\ 0&\mu&0\\ 0&0&\nu\end{pmatrix},\quad
\begin{pmatrix}\lambda&1&0\\ 0&\lambda&0\\ 0&0&\mu\end{pmatrix},\quad
\begin{pmatrix}1&1&0\\ 0&1&1\\ 0&0&1\end{pmatrix},\quad
\begin{pmatrix}
\lambda & 0 & 0 \\
0 & e^{\alpha}\cos\beta & -e^{\alpha}\sin\beta \\
0 & e^{\alpha}\sin\beta & e^{\alpha}\cos\beta
\end{pmatrix},
\]
with $\lambda,\mu,\nu>0$ pairwise distinct.
\end{Remark}

The $G$-action on $\mathbb{RP}^{\,n-1}$ lifts to a $\tilde{G}$-action on the universal cover $S^{\,n-1}$. If
\[
\gamma(\kappa(x))=h\cdot \gamma(x),
\]
and $\Gamma$ is a lift of $\gamma$ to $\widetilde{\mathbb{RP}}{}^{\,n-1}$, then
\[
\Gamma(\kappa(x))=\tilde{h}\cdot \Gamma(x).
\]
Consequently,
\[
\begin{tikzcd}
\mathcal{D}^G_n(\mathbf{C}) \arrow[r, "{\widetilde{\mathcal{D}_{q}^{\prime}}}"] \arrow[rd, "\mathcal{D}_{q}^{\prime}"'] & \tilde{G} \arrow[d, "p"] \\
& G
\end{tikzcd}
\]

where $\widetilde{\mathcal{D}^{'}_q}$ is the lifted monodromy and $p$ is the universal covering map.
\begin{Proposition}\label{Proposition 2.2}
$\mathcal{D}^{'}_q$ is a submersion.
\end{Proposition}

\begin{proof}
Assume $\mathbf{C}=S^1$, so $\kappa(x)=x+1$. We produce local sections. Let $\mathfrak{g}=\mathrm{Lie}(G)$ and $\exp:\mathfrak{g}\to G$ be the exponential map. If $h$ is in the image of $\mathcal{D}^{'}_q$, choose $\gamma$ with $\gamma(x+1)=h\cdot \gamma(x)$. For a small $\xi:\mathbb{R}\to \mathfrak{g}$, set $\gamma_\xi(x)=\exp(\xi(x))\cdot \gamma(x)$ and look for $h_\xi$ close to $h$ with
\[
\gamma_\xi(x+1)=h_\xi\cdot \gamma_\xi(x).
\]
Using $\gamma(x+1)=h\cdot \gamma(x)$, this reads
\[
\exp(\xi(x+1))=h_\xi\,\exp(\xi(x))\,h^{-1}.
\]
For $h_\xi=h\,\exp(\eta)$ with small $\eta\in\mathfrak{g}$, and since $\exp$ is a local diffeomorphism near $0$, one can solve
\[
\xi(x+1)=\eta+\operatorname{Ad}_h \xi(x)
\]
for $\xi$ (e.g. by choosing $\xi$ on $[0,1)$ and defining it inductively), yielding a family $\gamma_\xi$ with monodromy $h_\xi$. Non-degeneracy is open, so this gives local sections through $h$.
\end{proof}

Since $p:\tilde{G}\to G$ is a local diffeomorphism, Proposition \ref{Proposition 2.2} implies $\widetilde{\mathcal{D}^{'}_q}$ is a submersion. Its image is thus open; denote it by $\tilde{G}_+$.
\subsection{Action of $\mathrm{Diff}_+(\mathbf{C})$}
\begin{Notation}
$\mathrm{Diff}_{\mathbb{Z}}(\tilde{\mathbf{C}})$ is the group of diffeomorphisms of $\tilde{\mathbf{C}}$ that commute with $\kappa$. Let $\mathrm{Diff}_+(\mathbf{C})$ be the group of orientation-preserving diffeomorphisms of $\mathbf{C}$.
\end{Notation}
The group $\mathrm{Diff}_{\mathbb{Z}}(\tilde{\mathbf{C}})$ is the universal cover of $\mathrm{Diff}_+(\mathbf{C})$. It acts on $\gamma \in \mathcal{D}^G_n(\mathbf{C})$ by
\[
(F \cdot \gamma)(x) = \gamma(F^{-1}(x)).
\]
Moreover, $\mathrm{Diff}_+(\mathbf{C})$ acts on $r$-densities by pullback/pushforward. In local coordinates,
\[
F^*: \ \phi(x)\,|\partial x|^r \longmapsto \big(F'(x)\big)^r\,\phi(F(x))\,|\partial x|^r.
\]
Infinitesimally, for $v=h(x)\,\frac{d}{dx}\in\mathrm{Vect}(\mathbf{C})$,
\[
\mathcal{L}_v\big(\phi\,|\partial x|^r\big)=\big(h\,\phi'+r\,h'\phi\big)\,|\partial x|^r.
\]
This induces an action on $L\in \mathcal{R}^G_n(\mathbf{C})$ by
\[
F\cdot L=F_* \circ L \circ F^*,\qquad F\in \mathrm{Diff}_+(\mathbf{C}),
\]
and the infinitesimal action is $[\mathcal{L}_v,L]$.
\begin{Lemma}
Let $L \in \mathcal{R}_n(\mathbf{C})$ be an $n$-th order operator of the form
\[
L=\frac{d^n}{d\theta^n}+\cdots+a_1(\theta)\frac{d}{d\theta}+a_0(\theta).
\]
If $F^{-1}\in \mathrm{Diff}_+(\mathbf{C})$ acts on $L$, then the transformed operator
\[
\tilde{L}=\frac{d^n}{d\theta^n}+\cdots+\tilde{a}_1(\theta)\frac{d}{d\theta}+\tilde{a}_0(\theta)
\]
satisfies the following transformation rules:
\begin{itemize}
    \item For $n=2$,
    \[
    \tilde{a}_0=F^*a_0+\mathcal{S}(F).
    \]
    \item For $n=3$,
    \[
    \tilde{a}_1=(F')^2\,F^*a_1+2\,\mathcal{S}(F), \qquad
    \tilde{a}_0=(F')^3\,F^*a_0+F'F''\,F^*a_1+\mathcal{S}(F)',
    \]
\end{itemize}
where the Schwarzian derivative of $F$ is
\[
\mathcal{S}(F)=\frac{F'''}{F'}-\frac{3}{2}\left(\frac{F''}{F'}\right)^2.
\]
\end{Lemma}
\begin{Notation}
For $\gamma\in \mathcal{D}_n^G(\mathbf{C})$, let $\big(\mathrm{Diff}_{\mathbb{Z}}(\tilde{\mathbf{C}})\big)_\gamma$ be its stabilizer.
\end{Notation}
\begin{Lemma}
If $F \in \big(\mathrm{Diff}_{\mathbb{Z}}(\tilde{\mathbf{C}})\big)_\gamma$, then either $F$ has no fixed points on $\tilde{\mathbf{C}}$, or $F = \mathrm{Id}$.
\end{Lemma}
\begin{proof}
Assume $\mathbf{C}=S^1$ and $\tilde{\mathbf{C}}=\mathbb{R}$. Since $\gamma$ is an immersion, for each $x\in\mathbb{R}$ there exists an open interval $I_x$ on which $\gamma$ is an embedding. If $\alpha\in I_x$ and $\beta=F(\alpha)\in I_x$, then
\[
\gamma(\beta)=\gamma(F(\alpha))=\gamma(\alpha).
\]
Injectivity on $I_x$ yields $\beta=\alpha$. Hence $I_x\cap F(I_x)=\{\alpha\in I_x\mid F(\alpha)=\alpha\}$ is open and closed in $I_x$, so either $I_x\cap F(I_x)=I_x$ or $=\varnothing$. Therefore the fixed point set
\[
S=\{\alpha\in\mathbb{R}\mid F(\alpha)=\alpha\}
\]
is open and closed. If $S\neq\varnothing$, then $S=\mathbb{R}$ and $F=\mathrm{Id}$.
\end{proof}

\begin{Lemma}
$\big(\mathrm{Diff}_{\mathbb{Z}}(\tilde{\mathbf{C}})\big)_\gamma$ acts properly discontinuously on $\tilde{\mathbf{C}}$.
\end{Lemma}

\begin{proof}
As above, for each $x$ pick an interval $I_x$ on which $\gamma$ is an embedding. If $F\neq \mathrm{Id}$ is in the stabilizer, then $I_x\cap F(I_x)=\varnothing.$
\end{proof}

\begin{Proposition}
The map $\mathcal{D}_p:\mathcal{D}^G_n(\mathbf{C})\to \mathcal{R}^G_n(\mathbf{C})$ is $\mathrm{Diff}_{\mathbb{Z}}(\tilde{\mathbf{C}})$-equivariant.
\end{Proposition}

\begin{proof}
This is \cite[Thm.~2.2.3]{ovsienko2004projective}.
\end{proof}
\section{Constructions via Drinfeld-Sokolov method}

In this section, we first discuss a Morita equivalence between a symplectic groupoid integrating the space of connections on a principal $G$-bundle $Q \rightarrow \mathbf{C}$ and the quasi-symplectic groupoid that integrates the Cartan-Dirac structure on $\widetilde{G}$, with the path space $\mathcal{P}(Q)$ as a bimodule. Then, using the Drinfeld-Sokolov reduction \cite{drinfeld1985lie}, We reduce this Morita equivalence to one between the symplectic groupoid integrating the Adler-Gelfand-Dikii Poisson structure on $\mathcal{R}^G_n(\mathbf{C})$ and the quasi-symplectic groupoid integrating the Dirac structure on $Y_n(\mathbf{C})$. The reference for the first part is \cite{alekseev2009atiyah,alekseev2024coadjoint}. For the second part, see \cite{alekseev2024coadjoint,drinfeld1985lie}.

\subsection{Adler-Gelfand-Dikii Poisson structure}
This part introduces the \emph{Adler-Gelfand-Dikii} Poisson structure on $\mathcal{R}^G_n(\mathbf{C})$. The main reference is \cite[App.~A.7]{ovsienko2004projective}. Assume that $\theta$ is a coordinate on $\mathbf{C}$ and
\[
L \;=\; \frac{d^n}{d\theta^n} \;+\; a_{n-1}(\theta)\frac{d^{n-1}}{d\theta^{n-1}} \;+\; \cdots \;+\; a_0(\theta),
\]
is an element of $\mathcal{R}_n(\mathbf{C})$, where each $a_i \in C^\infty(\mathbf{C},\mathbb{R})$ is a smooth function. Suppose
\[
X \;=\; \sum_{j=1}^{n} X_j(\theta) \Bigl(\frac{d}{d\theta}\Bigr)^{-j}.
\]
Then, using the Leibniz rule,
\[
\Bigl(\frac{d}{d\theta}\Bigr)^{-1} \circ a(\theta) \;=\; a(\theta)\Bigl(\frac{d}{d\theta}\Bigr)^{-1} \;+\; \sum_{i=1}^\infty (-1)^i a^{(i)}(\theta) \Bigl(\frac{d}{d\theta}\Bigr)^{-i-1}.
\]
We can define
\[
X \cdot L \;=\; \sum_{m\in\mathbb{Z}} p_m \Bigl(\frac{d}{d\theta}\Bigr)^m.
\]
Using $X$, define a functional $\ell_X$ on $\mathcal{R}_n(\mathbf{C})$ by
\[
\ell_X(L) \;=\; \operatorname{res}(X \cdot L) \;=\; \int_{\mathbf{C}} p_{-1}(\theta)\, d\theta.
\]
It turns out that every continuous linear functional on $\mathcal{R}_n(\mathbf{C})$ is of the form $\ell_X$ for some pseudodifferential operator $X$. The Hamiltonian operator of the Adler-Gelfand-Dikii Poisson structure
\[
\pi^\sharp \colon T^*\mathcal{R}_n(\mathbf{C}) \longrightarrow T\,\mathcal{R}_n(\mathbf{C})
\]
is given by
\[
\ell_X \longmapsto V_X, \qquad V_X(L) \;=\; L(XL)_+ - (LX)_+ L \quad \text{for } L \in \mathcal{R}_n(\mathbf{C}).
\]
For a classical Lie group $G$, we restrict to those pseudodifferential symbols $X$ with the property that
\[
\frac{d^n}{d\theta^n} + V_X(L) \;\in\; \mathcal{R}^G_n(\mathbf{C}).
\]

\begin{Proposition}[\cite{ovsienko2004projective}, App.~A]
Symplectic leaves of the Adler-Gelfand-Dikii Poisson structure are in one-to-one correspondence with homotopy classes of non-degenerate curves with fixed monodromy.
\end{Proposition}

\subsection{Principal bundles over $\mathbf{C}$}
Let $G$ be a connected Lie group and $Q \to \mathbf{C}$ a principal $G$-bundle. Define
\[
\mathcal{P}(Q) \;=\; \{\tau \in \Gamma(\pi^*Q)\mid \exists\, h\in G \text{ such that } \kappa^*\tau = h\cdot \tau\}.
\]
\begin{Notation}
We denote the space of connections on $Q$ by $\mathcal{A}(Q)$ and the group of gauge transformations by $\mathrm{Gau}(Q)$.
\end{Notation}
We have the diagram
\[
\begin{tikzcd}
& \mathcal{P}(Q) \arrow[ld, "\mathcal{P}_p"'] \arrow[rd, "\mathcal{P}_q"] & \\
\mathcal{A}(Q) & & G
\end{tikzcd}
\]
where $\mathcal{P}_q\colon \mathcal{P}(Q)\to G$ maps $\tau$ to $h$, and $\mathcal{P}_p\colon \mathcal{P}(Q)\to \mathcal{A}(Q)$ maps $\tau$ to the connection whose pullback has $\tau$ as a horizontal section. The maps $\mathcal{P}_p$ and $\mathcal{P}_q$ are quotient maps for the commuting actions of $\mathrm{Gau}(Q)$ and $G$ on $\mathcal{P}(Q)$, respectively. If $G$ is not simply connected, then $\mathrm{Gau}(Q)$ is not connected; in this case we use its identity component $\mathrm{Gau}_0(Q)$. We then have
\[
\begin{tikzcd}
& \mathcal{P}(Q) \arrow[ld, "\mathcal{P}_p"'] \arrow[rd, "\widetilde{\mathcal{P}_q}"] & \\
\mathcal{A}(Q) & & \widetilde{G}
\end{tikzcd}
\]
where $\widetilde{\mathcal{P}_q}$ is the lifted monodromy.

\subsection{Poisson structure on $\mathcal{A}(Q)$}
Let $\mathfrak{g}=\mathrm{Lie}(G)$. The affine space $\mathcal{A}(Q)$ has underlying linear space $\Omega^1(\mathbf{C}, Q\times_G\mathfrak{g})$. The group $\mathrm{Gau}(Q)$ acts on $\mathcal{A}(Q)$ by pullback. The infinitesimal action of $\mathfrak{gau}(Q)=\Omega^0(\mathbf{C}, Q\times_G\mathfrak{g})$ at $A\in\mathcal{A}(Q)$ is
\[
\partial_A \colon \Omega^0(\mathbf{C}, Q\times_G\mathfrak{g}) \to \Omega^1(\mathbf{C}, Q\times_G\mathfrak{g}).
\]
Choose an invariant metric $\langle\cdot,\cdot\rangle$ on $\mathfrak{g}$. We identify $\mathfrak{gau}^*(Q)=\Omega^1(\mathbf{C}, Q\times_G\mathfrak{g})$ as the smooth dual. The action of \(\mathrm{Gau}(Q)\) on \(\mathcal{A}(Q)\) is affine with underlying linear coadjoint action.

\begin{Remark}[\cite{alekseev2024coadjoint}, §2.2]\label{rem:affine}
Given an affine action of a Lie group $K$ on an affine space $E$, with the coadjoint action as the underlying linear $K$-action on $\mathfrak{k}^*$, there is a central extension
\[
0 \to \mathbb{R} \to \widehat{\mathfrak{k}} \to \mathfrak{k} \to 0
\]
such that $E \cong \widehat{\mathfrak{k}}^{*}_1$ (the level-one affine hyperplane). Moreover
\[
\widehat{\mathfrak{k}} = \operatorname{Aff}(E,\mathbb{R}), \qquad
[\widehat{X},\widehat{Y}](\mu) = -\langle \operatorname{ad}^*_X \mu, Y\rangle,
\]
for $\widehat{X},\widehat{Y}\in \operatorname{Aff}(E,\mathbb{R})$ and $X\in \operatorname{Hom}(E,\mathbb{R})=\mathfrak{k}$.
\end{Remark}

Applying Remark~\ref{rem:affine} with $E=\mathcal{A}(Q)$, we obtain the central extension
\[
0 \to \mathbb{R} \to \widehat{\mathfrak{gau}}(Q) \to \mathfrak{gau}(Q) \to 0,
\]
with Lie bracket
\[
[\widehat{\xi}_1,\widehat{\xi}_2](A) \;=\; \int_{\mathbf{C}} \langle \partial_A \xi_1, \xi_2\rangle,
\]
and $\mathcal{A}(Q) \cong \widehat{\mathfrak{gau}}^{*}_1(Q)$. Therefore, $\mathcal{A}(Q)$ has a natural Poisson structure \cite[Ex.~A.12]{alekseev2024coadjoint}. The moment map for the $\mathrm{Gau}(Q)$-action is the inclusion
\[
\widehat{\mathfrak{gau}}^{*}_1(Q) \hookrightarrow \widehat{\mathfrak{gau}}^*(Q).
\]

\subsection{The 2-form \texorpdfstring{$\varpi_{\mathcal{P}}$}{varpiP}}
There exists a natural $2$-form $\varpi_{\mathcal{P}}$ on $\mathcal{P}(Q)$ with the following properties.

\begin{Theorem}[\cite{alekseev2024coadjoint}, Thm.~5.3]\label{thm:varpiP}
\begin{enumerate}
\item $d\varpi_{\mathcal{P}} = \widetilde{\mathcal{P}_q}^{*}\eta$, where $\eta = -\tfrac{1}{12}\,\langle \theta^L,[\theta^L,\theta^L]\rangle$ is the Cartan $3$-form on $\widetilde{G}$.
\item $\varpi_{\mathcal{P}}$ is $G$-invariant, and its contractions with generating vector fields for the $G$-action satisfy
\[
\iota_{X_{\mathcal{P}}}\varpi_{\mathcal{P}} \;=\; -\tfrac{1}{2}\,\widetilde{\mathcal{P}_q}^{*}\langle \theta^L+\theta^R, X\rangle.
\]
\item \label{it:varpiP-c} $\varpi_{\mathcal{P}}$ is $\mathrm{Gau}(Q)$-invariant, and for $\xi\in \mathfrak{gau}(Q)$,
\[
\iota_{\xi_{\mathcal{P}}}\varpi_{\mathcal{P}} \;=\; -\,\mathcal{P}_p^{*}\langle dA,\xi\rangle.
\]
In fact, $\varpi_{\mathcal{P}}$ is $\mathrm{Aut}_+(Q)$-invariant.
\end{enumerate}
\end{Theorem}

\begin{proof}
See \cite[Thm.~5.3]{alekseev2024coadjoint}. For $Q=S^1\times G$, see \cite[§8.4]{alekseev2009atiyah}.
\end{proof}

We have the following Morita equivalence of quasi-symplectic groupoids:
\begin{equation}\label{eq:morita-big}
\begin{tikzcd}
{(\mathrm{Gau}_0(Q)\ltimes \mathcal{A}(Q),\,\omega_1)} \arrow[d, shift right] \arrow[d, shift left] 
& {(\mathcal{P}(Q),\,\varpi_{\mathcal{P}})} \arrow[ld, "\mathcal{P}_p"'] \arrow[rd, "\widetilde{\mathcal{P}_q}"] 
& {(G\ltimes \widetilde{G},\,\omega_2)} \arrow[d, shift left] \arrow[d, shift right] \\
\mathcal{A}(Q) & & \widetilde{G}
\end{tikzcd}
\end{equation}
where $\mathrm{Gau}_0(Q)\ltimes \mathcal{A}(Q) = (\mathcal{P}(Q)\times_{\widetilde{G}}\mathcal{P}(Q))/G$, and $\omega_1$ is induced by $pr_1^*\varpi_{\mathcal{P}} - pr_2^*\varpi_{\mathcal{P}}$. The quasi-symplectic groupoid $(G\ltimes \widetilde{G},\omega_2)\rightrightarrows \widetilde{G}$ integrates the Cartan-Dirac structure on $\widetilde{G}$. If $\mathbf{C}=S^1$, then $Q\cong S^1\times G$ and
\[
\mathcal{P}(Q)=\{\beta\colon \mathbb{R}\to G \mid \beta(t+1)=h\,\beta(t)\ \text{for some }h\in G\}.
\]
In this case $\mathcal{A}(Q)=\Omega^1(S^1,\mathfrak{g})$ and $\mathrm{Gau}(Q)=C^\infty(S^1,G)$.

\subsection{Drinfeld-Sokolov embedding}
\begin{Notation}
Unless stated otherwise, assume $G=\mathrm{PSL}(n,\mathbb{R})$.
\end{Notation}
Consider the rank-$n$ bundle $E=J^{n-1}(|\Lambda|_{\mathbf{C}}^{-\frac{n-1}{2}})$. It defines the projective bundle $\mathbf{P}:=\mathbb{P}(E)\to \mathbf{C}$ and hence a principal $\mathrm{PSL}(n,\mathbb{R})$-bundle $Q\to \mathbf{C}$. A lift at any $x\in \widetilde{\mathbf{C}}$ defines an $\mathrm{SL}(n,\mathbb{R})$-frame of $E_x$, unique up to sign, hence an element of $\pi^*Q_x$. For a quasi-periodic curve $\gamma$ we obtain a quasi-periodic section of $\pi^*Q$, i.e.\ an element of $\mathcal{P}(Q)$. This defines
\begin{equation}\label{eq:iota-hat}
\hat{\iota}\colon \mathcal{D}_n(\mathbf{C})\longrightarrow \mathcal{P}(Q).
\end{equation}
If $\mathbf{C}=S^1$, we have the trivialization
\[
E_x=J_x^{n-1}\!\bigl(|\Lambda|_{\mathbf{C}}^{-\frac{n-1}{2}}\bigr) \longrightarrow \mathbb{R}^n,\qquad
j^{n-1}(u)\longmapsto \begin{pmatrix} u(x)\\ \vdots\\ u^{(n-1)}(x)\end{pmatrix},
\]
and
\[
\hat{\iota}(\gamma) \;=\; \tau \;=\;
\begin{pmatrix}
u_1 & \cdots & u_n\\
\vdots & & \vdots \\
u_1^{(n-1)} & \cdots & u_n^{(n-1)}
\end{pmatrix}.
\]
The action of $\mathrm{Diff}_+(\mathbf{C})$ on densities defines a bundle automorphism of $E$ and a principal bundle automorphism of $Q\to \mathbf{C}$, giving a group homomorphism
\[
\mathrm{Diff}_+(\mathbf{C}) \longrightarrow \mathrm{Aut}_+(Q),
\]
which splits the exact sequence
\[
1 \to \mathrm{Gau}(Q) \to \mathrm{Aut}_+(Q) \to \mathrm{Diff}_+(\mathbf{C}) \to 1.
\]
Thus, $\mathrm{Aut}_+(Q)=\mathrm{Diff}_+(\mathbf{C})\ltimes \mathrm{Gau}(Q)$. The map $\hat{\iota}$ is $G$-equivariant and descends to
\begin{equation}\label{eq:iota}
\iota\colon \mathcal{R}_n(\mathbf{C})\longrightarrow \mathcal{A}(Q).
\end{equation}

\begin{Remark}
Via ODE theory, $L$ defines a connection $\nabla$ on $E$ with horizontal sections $j^{n-1}(u_1),\dots,j^{n-1}(u_n)$, hence a connection on $Q$. This gives the same map as \eqref{eq:iota}.
\end{Remark}

If $\mathbf{C}=S^1$, every element of $\mathcal{R}_n(S^1)$ is an $n$-th order linear ODE with principal symbol \(1\). If
\[
L=\frac{d^n}{d\theta^n}+a_{n-2}(\theta)\frac{d^{n-2}}{d\theta^{n-2}}+\cdots + a_0(\theta),
\]
then under $\iota$ it corresponds to the companion matrix:

\[
\begin{pmatrix}
0   & 0   & 0   & 0   & \cdots & 0      & -a_0 \\
1   & 0   & 0   & 0   & \cdots & 0      & -a_1 \\
0   & 1   & 0   & 0   & \cdots & 0      & -a_2 \\
0   & 0   & 1   & 0   & \cdots & 0      & -a_3 \\
\vdots & \vdots & \vdots & \ddots & \ddots & \vdots & \vdots \\
0   & 0   & 0   & 0   & \cdots & 0      & -a_{n-2} \\
0   & 0   & 0   & 0   & \cdots & 1      & 0
\end{pmatrix}.
\]
\subsection{The 2-form \texorpdfstring{$\varpi_{\mathcal{D}}$}{varpiD}}
Define $\varpi_{\mathcal{D}}:=\hat{\iota}^{*}\varpi_{\mathcal{P}}$ on $\mathcal{D}_n(\mathbf{C})$. Then:
\begin{Theorem}\label{thm:varpiD}
\begin{enumerate}
\item $d\varpi_{\mathcal{D}} = \widetilde{\mathcal{D}_q^{'}}^{*}\eta$.
\item $\varpi_{\mathcal{D}}$ is $G$-invariant, and for $X\in\mathfrak{gl}(n,\mathbb{R})$,
\[
\iota_{X_{\mathcal{D}}}\varpi_{\mathcal{D}} \;=\; 
\tfrac{1}{2}\,\operatorname{tr}\,\Bigl(X\,\widetilde{\mathcal{D}_q^{'}}^{*}(\theta^L+\theta^R)\Bigr).
\]
\item $\varpi_{\mathcal{D}}$ is $\mathrm{Diff}_{\mathbb{Z}}(\widetilde{\mathbf{C}})$-invariant, and for $v\in \mathrm{Vect}(\mathbf{C})$,
\[
\iota_{v_{\mathcal{D}}}\varpi_{\mathcal{D}} \;=\; 
-\,\mathcal{D}_p^{*}\!\left(\int_{\mathbf{C}} (dL)\,v \right).
\]
\end{enumerate}
\end{Theorem}

\begin{proof}
(a) Consider
\[
\begin{tikzcd}
\mathcal{D}_n(\mathbf{C}) \arrow[r, "\hat{\iota}"] \arrow[rd, "\widetilde{\mathcal{D}_q^{'}}"'] & \mathcal{P}(Q) \arrow[d, "\widetilde{\mathcal{P}_q}"] \\
& \widetilde{G}
\end{tikzcd}
\]
and use $d\varpi_{\mathcal{D}}=\hat{\iota}^{*}(d\varpi_{\mathcal{P}})=\hat{\iota}^{*}\widetilde{\mathcal{P}_q}^{*}\eta=\widetilde{\mathcal{D}_q^{'}}^{*}\eta$ by Theorem~\ref{thm:varpiP}. The other parts follow from $G$-equivariance and the commutative diagram
\[
\begin{tikzcd}
\mathcal{D}_n(\mathbf{C}) \arrow[r, "\hat{\iota}"] \arrow[d, "\mathcal{D}_p"'] & \mathcal{P}(Q) \arrow[d, "\mathcal{P}_p"] \\
\mathcal{R}_n(\mathbf{C}) \arrow[r, "\iota"] & \mathcal{A}(Q)
\end{tikzcd}
\]
together with Theorem~\ref{thm:varpiP}\ref{it:varpiP-c}.
\end{proof}

\subsection{Drinfeld-Sokolov reduction}
\begin{Notation}
We denote the vertical bundle by $V\mathbf{P}\subseteq T\mathbf{P}$.
\end{Notation}
Fix $\sigma\in \Gamma(\mathbf{P})$. Let $\mathcal{B}\subset \mathrm{Gau}(Q)$ preserve $\sigma$, and let $\mathcal{N}\subset \mathcal{B}$ act trivially on $\sigma^*V\mathbf{P}$. Then
\[
\mathcal{N}\subset \mathcal{B}\subset \mathcal{G}=\mathrm{Gau}(Q).
\]
Choose a metric on $\mathfrak{g}=\mathfrak{gl}(n,\mathbb{R})$. It induces a bundle metric on $Q\times_G\mathfrak{g}$ and a non-degenerate $C^\infty(\mathbf{C})$-valued bilinear form on $\mathfrak{gau}(Q)=\Omega^0(\mathbf{C},Q\times_G\mathfrak{g})$, giving a non-degenerate pairing between $\mathrm{Lie}(\mathcal{G})/\mathrm{Lie}(\mathcal{B})$ and $\mathrm{Lie}(\mathcal{N})$. Using local trivializations, $\mathrm{Lie}(\mathcal{G})/\mathrm{Lie}(\mathcal{B})\cong \Gamma(\sigma^*V\mathbf{P})$. Integrating over $\mathbf{C}$,
\begin{equation}
\Omega^1(\mathbf{C},\sigma^*V\mathbf{P}) \;\cong\; \mathrm{Lie}(\mathcal{N})^*,
\end{equation}
where ${}^*$ denotes the smooth dual.

With $\mathbf{P}=\mathbb{P}(E)$ and $E=J^{n-1}(|\Lambda|_{\mathbf{C}}^{-\frac{n-1}{2}})$, sections of $\mathbf{P}$ correspond to line subbundles of $E$. From the exact sequence
\[
0 \to \operatorname{Sym}^{n-1}(T^*\mathbf{C})\otimes |\Lambda|_{\mathbf{C}}^{-\frac{n-1}{2}} \to J^{n-1}\!\bigl(|\Lambda|_{\mathbf{C}}^{-\frac{n-1}{2}}\bigr) \to J^{n-2}\!\bigl(|\Lambda|_{\mathbf{C}}^{-\frac{n-1}{2}}\bigr) \to 0,
\]
we get the line bundle
\[
\ell \;=\; \operatorname{Sym}^{n-1}(T^*\mathbf{C})\otimes |\Lambda|_{\mathbf{C}}^{-\frac{n-1}{2}} \;=\; |\Lambda|_{\mathbf{C}}^{\frac{n-1}{2}}.
\]
Therefore
\[
\sigma^*V\mathbf{P} \;=\; \ell^*\otimes (E/\ell) \;=\; |\Lambda|_{\mathbf{C}}^{-\frac{n-1}{2}} \otimes J^{n-2}\!\bigl(|\Lambda|_{\mathbf{C}}^{-\frac{n-1}{2}}\bigr).
\]
Identify $\mathcal{A}(Q)$ with projective connections on $\mathbf{P}$. There is an $\mathcal{N}$-equivariant map
\[
\Psi\colon \mathcal{A}(Q)\to \mathrm{Lie}(\mathcal{N})^*,
\]
which sends $A\in \mathcal{A}(Q)$ to the composition $T\sigma\colon T\mathbf{C}\to T\mathbf{P}$ followed by the vertical projection defined by $A$. If $\mathbf{C}=S^1$, using the Iwasawa decomposition $G=KAN$ with $B=AN$, one has $\mathcal{B}=C^\infty(S^1,B)$, $\mathcal{N}=C^\infty(S^1,N)$, $\mathcal{A}(Q)=\Omega^1(S^1,\mathfrak{g})\cong C^\infty(S^1,\mathfrak{g})$, and $\Psi$ projects to the strictly lower triangular part.

\begin{Theorem}[Drinfeld-Sokolov]\label{thm:DS}
There exists $\Lambda\in \mathrm{Lie}(\mathcal{N})^{*}$, invariant for the coadjoint $\mathcal{N}$-action, such that the inclusion
\[
\mathcal{R}_n(\mathbf{C}) \longrightarrow \mathcal{A}(Q)
\]
takes values in $\Psi^{-1}(\Lambda)$ and is a global slice for the $\mathcal{N}$-action on this level set. Consequently,
\[
\mathcal{R}_n(\mathbf{C}) \;\cong\; \Psi^{-1}(\Lambda)/\mathcal{N}
\]
as Poisson manifolds.
\end{Theorem}

\begin{proof}
For $\mathbf{C}=S^1$ take
\[
\Lambda =
\begin{pmatrix}
0 & 0 & 0 & \cdots & 0 & 0 \\
1 & 0 & 0 & \cdots & 0 & 0 \\
0 & 1 & 0 & \cdots & 0 & 0 \\
\vdots & \vdots & \ddots & \ddots & \vdots & \vdots \\
0 & 0 & \cdots & 1 & 0 & 0 \\
0 & 0 & \cdots & 0 & 1 & 0
\end{pmatrix}.
\]
Then
\[
\Psi^{-1}(\Lambda) =
\begin{pmatrix}
* & * & * & \cdots & * & * \\
1 & * & * & \cdots & * & * \\
0 & 1 & * & \cdots & * & * \\
0 & 0 & 1 & \ddots & \vdots & \vdots \\
\vdots & \vdots & \ddots & \ddots & * & * \\
0 & 0 & \cdots & 0 & 1 & *
\end{pmatrix}.
\]
In every $\mathcal{N}$-orbit there is a unique element of the form
\[
\begin{pmatrix}
0   & 0   & 0   & 0   & \cdots & 0      & -a_0 \\
1   & 0   & 0   & 0   & \cdots & 0      & -a_1 \\
0   & 1   & 0   & 0   & \cdots & 0      & -a_2 \\
0   & 0   & 1   & 0   & \cdots & 0      & -a_3 \\
\vdots & \vdots & \vdots & \ddots & \ddots & \vdots & \vdots \\
0   & 0   & 0   & 0   & \cdots & 0      & -a_{n-2} \\
0   & 0   & 0   & 0   & \cdots & 1      & 0
\end{pmatrix}.
\]
See \cite[§6.5]{drinfeld1985lie}, \cite[App.~A.8]{khesin2009geometry}, and \cite[§9.4]{dickey2003soliton}.
\end{proof}

Consider $\hat{\Psi}=\Psi\circ \mathcal{P}_p\colon \mathcal{P}(Q)\to \mathrm{Lie}(\mathcal{N})^*$. By Theorem~\ref{thm:DS}, $\hat{\iota}\colon \mathcal{D}_n(\mathbf{C})\to \mathcal{P}(Q)$ is a global slice for the $\mathcal{N}$-action on $\hat{\Psi}^{-1}(\Lambda)$. Part~\ref{it:varpiP-c} of Theorem~\ref{thm:varpiP} implies
\begin{equation}\label{eq:contraction-Psi}
\iota_{\xi_{\mathcal{P}(Q)}}\varpi_{\mathcal{P}} \;=\; -\langle d\hat{\Psi}, \xi\rangle.
\end{equation}
Hence $\varpi_{\mathcal{P}}|_{\hat{\Psi}^{-1}(\Lambda)}$ is $\mathcal{N}$-basic and descends to a $2$-form on $\mathcal{D}_n(\mathbf{C})$, which equals $\varpi_{\mathcal{D}}$ because $\hat{\iota}^{*}\varpi_{\mathcal{P}}=\varpi_{\mathcal{D}}$.

\subsection{A Morita equivalence for Adler-Gelfand-Dikii Poisson structures}
By definition, there exists a $G$-equivariant map 
\[ f:Y_n(\mathbf{C}) \to \widetilde{G} \] 
that pulls back the Cartan-Dirac structure to a Dirac structure on $Y_n(\mathbf{C})$. 
Moreover, the following diagram commutes:
\[
\begin{tikzcd}
\mathcal{D}^G_n(\mathbf{C}) \arrow[d, "\widetilde{\mathcal{D}}_q"'] \arrow[rd, "\widetilde{\mathcal{D}}_q^{\prime}"] & \\
Y_n(\mathbf{C}) \arrow[r, "f"'] & \widetilde{G}
\end{tikzcd}
\]
We have a Lie groupoid homomorphism given by
\begin{equation}\label{eq:diagram-46}
\begin{tikzcd}
\mathcal{D}_n(\mathbf{C}) \times_{Y_n(\mathbf{C})} \mathcal{D}_n(\mathbf{C}) \arrow[r, "{(\hat{\iota},\hat{\iota})}"] \arrow[d, shift right] \arrow[d, shift left] & \mathcal{P}(Q)\times_{\widetilde{G}} \mathcal{P}(Q) \arrow[d, shift left] \arrow[d, shift right] \\
\mathcal{D}_n(\mathbf{C}) \arrow[r, "\hat{\iota}"] & \mathcal{P}(Q)
\end{tikzcd}.
\end{equation}
The diagram \eqref{eq:diagram-46} is $G$-equivariant (diagonal actions), hence descends to
\[
\begin{tikzcd}
\mathcal{D}_n(\mathbf{C})\times_{Y_n(\mathbf{C})}\mathcal{D}_n(\mathbf{C})/G \arrow[r, "\tilde{\iota}"] \arrow[d, shift right] \arrow[d, shift left] & \mathcal{P}(Q)\times_{\widetilde{G}}\mathcal{P}(Q)/G \arrow[d, shift left] \arrow[d, shift right] \\
\mathcal{R}_n(\mathbf{C}) \arrow[r, "\iota"] & \mathcal{A}(Q)
\end{tikzcd}.
\]

The $2$-action of the Lie $2$-group $\mathcal{N}\times \mathcal{N} \rightrightarrows \mathcal{N}$ on $(\mathcal{P}(Q)\times_{\widetilde{G}}\mathcal{P}(Q)/G,\omega_1)\rightrightarrows \mathcal{A}(Q)$ has moment map
\[
\Phi \colon \mathcal{P}(Q)\times_{\widetilde{G}}\mathcal{P}(Q)/G \longrightarrow \mathrm{Lie}(\mathcal{N})^*\times \mathrm{Lie}(\mathcal{N})^*, \qquad \Phi=(\Psi,\Psi)\circ (s,t),
\]
and base moment map $\Psi\colon \mathcal{A}(Q)\to \mathrm{Lie}(\mathcal{N})^*$. The slice $\tilde{\iota}$ meets every $\mathcal{N}\times \mathcal{N}$-orbit in $\Phi^{-1}(\Lambda,\Lambda)$ once. Applying reduction (e.g.\ Proposition~\ref{Proposition 3.2}) yields a symplectic groupoid
\[
\begin{tikzcd}
(\mathcal{D}_n(\mathbf{C}) \times_{Y_n(\mathbf{C})} \mathcal{D}_n(\mathbf{C})/G, \,\omega_{\mathrm{red}}) \arrow[d, "s"', shift right] \arrow[d, "t", shift left] \\
\mathcal{R}_n(\mathbf{C})
\end{tikzcd}
\]
integrating the Adler-Gelfand-Dikii Poisson structure on $\mathcal{R}_n(\mathbf{C})$.

\begin{Lemma}
The $2$-form $pr_1^*\varpi_{\mathcal{D}} - pr_2^*\varpi_{\mathcal{D}}$ descends to a closed multiplicative $2$-form $\tilde{\omega}_1$ on $\bigl(\mathcal{D}_n(\mathbf{C})\times_{Y_n(\mathbf{C})}\mathcal{D}_n(\mathbf{C})\bigr)/G \rightrightarrows \mathcal{R}_n(\mathbf{C})$.
\end{Lemma}

\begin{proof}
By $G$-invariance of $\varpi_{\mathcal{D}}$, the form $pr_1^*\varpi_{\mathcal{D}} - pr_2^*\varpi_{\mathcal{D}}$ is $G$-basic. Moreover,
\[
d\!\left(pr_1^*\varpi_{\mathcal{D}} - pr_2^*\varpi_{\mathcal{D}}\right)
= pr_1^*\widetilde{\mathcal{D}_q}^{*}\eta - pr_2^*\widetilde{\mathcal{D}_q}^{*}\eta = 0,
\]
so it is closed. Multiplicativity holds upstairs and descends.
\end{proof}

Since $\tilde{\iota}^{*}\omega_1 = \tilde{\omega}_1$, we have $\omega_{\mathrm{red}}=\tilde{\omega}_1$. Reducing the Morita equivalence \eqref{eq:morita-big} gives
\begin{equation}\label{eq:morita-reduced}
\begin{tikzcd}
{(\mathcal{D}_n(\mathbf{C}) \times_{Y_n(\mathbf{C})} \mathcal{D}_n(\mathbf{C})/G,\, \tilde{\omega}_1)} \arrow[d, shift right] \arrow[d, shift left] 
& {(\mathcal{D}_n(\mathbf{C}),\, \varpi_{\mathcal{D}})} \arrow[ld, "\mathcal{D}_p"'] \arrow[rd, "\widetilde{\mathcal{D}_q}"] 
& {(G \ltimes Y_n(\mathbf{C}),\, \tilde{\omega}_2)} \arrow[d, shift right] \arrow[d, shift left] \\
\mathcal{R}_n(\mathbf{C}) & & Y_n(\mathbf{C})
\end{tikzcd}
\end{equation}

\begin{Theorem}\label{thm:morita-G}
Let $G$ be one of $\mathrm{PSL}(n)$, $\mathrm{PSp}(n)$ (even $n$), or $\mathrm{PSO}(n)$ (odd $n$). Then we have a Morita equivalence of quasi-symplectic groupoids:
\begin{equation}\label{eq:morita-G}
\begin{tikzcd}
{(\mathcal{D}^G_n(\mathbf{C}) \times_{Y_n(\mathbf{C})}\mathcal{D}^G_n(\mathbf{C})/G,\, \tilde{\omega}_1)} \arrow[d, shift right] \arrow[d, shift left] 
& {(\mathcal{D}^G_n(\mathbf{C}),\, \varpi_{\mathcal{D}})} \arrow[ld, "\mathcal{D}_p"'] \arrow[rd, "\widetilde{\mathcal{D}_q}"] 
& {(G \ltimes Y_n(\mathbf{C}),\, \tilde{\omega}_2)} \arrow[d, shift right] \arrow[d, shift left] \\
\mathcal{R}^G_n(\mathbf{C}) & & Y_n(\mathbf{C})
\end{tikzcd}
\end{equation}
\end{Theorem}

\begin{proof}
The case $G=\mathrm{PSL}(n)$ follows from \eqref{eq:morita-reduced}. For general $G$, restrict all objects to the $G$-submanifold $\mathcal{D}^G_n(\mathbf{C})\subset \mathcal{D}_n(\mathbf{C})$.
\end{proof}

\begin{Lemma}
There is a Lie groupoid morphism
\[
\begin{tikzcd}
\mathrm{Diff}_{\mathbb{Z}}(\widetilde{\mathbf{C}})\ltimes \mathcal{R}^G_n(\mathbf{C}) \arrow[r, "\phi"] \arrow[d, shift right] \arrow[d, shift left] & (\mathcal{D}^G_n(\mathbf{C}) \times_{Y_n(\mathbf{C})}\mathcal{D}^G_n(\mathbf{C})) / G \arrow[d, shift left] \arrow[d, shift right] \\
\mathcal{R}^G_n(\mathbf{C}) \arrow[r, "\mathrm{Id}"] & \mathcal{R}^G_n(\mathbf{C})
\end{tikzcd}
\]
mapping $(F,L)$ to $(\gamma, F\cdot \gamma)$, where $\gamma$ is a quasi-periodic non-degenerate curve corresponding to $L$.
\end{Lemma}

\begin{proof}
Consider
\begin{equation}\label{eq:phi-tilde}
\begin{tikzcd}
\mathrm{Diff}_{\mathbb{Z}}(\widetilde{\mathbf{C}})\ltimes \mathcal{D}^G_n(\mathbf{C}) \arrow[r, "\tilde{\phi}"] \arrow[d, shift right] \arrow[d, shift left] & \mathcal{D}^G_n(\mathbf{C}) \times_{Y_n(\mathbf{C})} \mathcal{D}^G_n(\mathbf{C}) \arrow[d, shift left] \arrow[d, shift right] \\
\mathcal{D}^G_n(\mathbf{C}) \arrow[r, "\mathrm{Id}"] & \mathcal{D}^G_n(\mathbf{C})
\end{tikzcd}
\end{equation}
with $\tilde{\phi}(F,\gamma)=(\gamma, F\cdot \gamma)$. The diagram \eqref{eq:phi-tilde} is $G$-equivariant and descends to
\[
\begin{tikzcd}
\mathrm{Diff}_{\mathbb{Z}}(\widetilde{\mathbf{C}})\ltimes \mathcal{R}^G_n(\mathbf{C}) \arrow[d, shift right] \arrow[d, shift left] \arrow[r, "\phi"] 
& (\mathcal{D}^G_n(\mathbf{C}) \times_{Y_n(\mathbf{C})} \mathcal{D}^G_n(\mathbf{C}))/G \arrow[d, shift left] \arrow[d, shift right] \\
\mathcal{R}^G_n(\mathbf{C}) \arrow[r, "\mathrm{Id}"] & \mathcal{R}^G_n(\mathbf{C})
\end{tikzcd}.
\]
\end{proof}
\begin{appendix}
\section{Reduction of symplectic Lie groupoids by pair Lie 2-group}
In this section, we discuss the reduction of symplectic groupoids under Hamiltonian actions of pair Lie 2-groups. Through this process, we obtain new symplectic groupoids.

\subsection{Quotient of Lie groupoids by Lie 2-groups}

\begin{Definition}[\cite{herrera2023isometric}, Section~2.2]
A \emph{Lie 2-group} is a group internal to the category of Lie groupoids. Equivalently, it is a Lie groupoid \( H_1 \rightrightarrows H_0 \) such that \( H_1 \) and \( H_0 \) are Lie groups and all structure maps of \( H_1 \rightrightarrows H_0 \) are Lie group homomorphisms.
\end{Definition}

\begin{Example}
Let \( H \) be a Lie group.
\begin{enumerate}
\item The Lie groupoid \( H \rightrightarrows H \) with source, target, and unit maps the identity is a Lie 2-group.
\item The pair groupoid \( H \times H \rightrightarrows H \) is a Lie 2-group.
\end{enumerate}
\end{Example}

\begin{Definition}[\cite{herrera2023isometric}, Section~2.2]
A \emph{Lie 2-algebra} is a Lie algebra internal to the category of Lie groupoids: a Lie groupoid \( \mathfrak{h}_1 \rightrightarrows \mathfrak{h}_0 \) where \( \mathfrak{h}_1,\mathfrak{h}_0 \) are Lie algebras and all structure maps are Lie algebra homomorphisms.
\end{Definition}

\begin{Remark}
If \( H_1 \rightrightarrows H_0 \) is a Lie 2-group, then by differentiation we obtain a Lie 2-algebra \( \mathrm{Lie}(H_1) \rightrightarrows \mathrm{Lie}(H_0) \).
\end{Remark}

\begin{Definition}[\cite{herrera2023isometric}, Section~3]
A \emph{left 2-action} of a Lie 2-group \( H_1 \rightrightarrows H_0 \) on a Lie groupoid \( \mathcal{G} \rightrightarrows M \) is a Lie groupoid homomorphism
\[
\mathcal{A}=(\mathcal{A}^1,\mathcal{A}^0):\ (H_1\times \mathcal{G}\rightrightarrows H_0\times M)\longrightarrow (\mathcal{G}\rightrightarrows M),
\]
from the product groupoid, such that \( \mathcal{A}^1 \) and \( \mathcal{A}^0 \) are usual left Lie group actions. We call the 2-action \emph{free} (resp.\ \emph{proper}) if both \( \mathcal{A}^1 \) and \( \mathcal{A}^0 \) are free (resp.\ proper).
\end{Definition}

\begin{Notation}
For a Lie 2-group \( H_1 \rightrightarrows H_0 \), write \( s_H,t_H,u_H \) for source, target, and unit. For a Lie groupoid \( \mathcal{G}\rightrightarrows M \), write \( s_{\mathcal{G}},t_{\mathcal{G}},u_{\mathcal{G}} \).
\end{Notation}

\begin{Remark}[\cite{herrera2023isometric}, Section~3]\label{3.2}
Given a left 2-action of \( H_1 \rightrightarrows H_0 \) on \( \mathcal{G} \rightrightarrows M \), we have:
\begin{enumerate}
\item \( s_{\mathcal{G}}(h_1\cdot g)= s_H(h_1)\cdot s_{\mathcal{G}}(g). \)
\item \( t_{\mathcal{G}}(h_1\cdot g)= t_H(h_1)\cdot t_{\mathcal{G}}(g). \)
\item \( u_{\mathcal{G}}(h_0\cdot m)= u_H(h_0)\cdot u_{\mathcal{G}}(m). \)
\end{enumerate}
for all \( h_1\in H_1,\,h_0\in H_0,\,g\in \mathcal{G},\,m\in M \).
\end{Remark}

\begin{Proposition}\label{Proposition 3.1}
If a Lie 2-group \( H_1 \rightrightarrows H_0 \) acts freely and properly on a Lie groupoid \( \mathcal{G} \rightrightarrows M \) by a left 2-action, then the quotient \( \mathcal{G}/H_1 \rightrightarrows M/H_0 \) is a Lie groupoid.
\end{Proposition}

\begin{proof}
Freeness and properness imply \( \mathcal{G}/H_1 \) and \( M/H_0 \) are smooth manifolds. Define
\[
s([g])=[s_{\mathcal{G}}(g)],\qquad t([g])=[t_{\mathcal{G}}(g)].
\]
By Remark~\ref{3.2}, these are well-defined. Item~(3) implies that \( u_{\mathcal{G}}:M\to \mathcal{G} \) descends to a unit map \( u:M/H_0\to \mathcal{G}/H_1 \). Define multiplication and inverse by
\[
[g_1]\cdot [g_2]=[g_1g_2]\quad \text{whenever } s([g_1])=t([g_2]),\qquad [g]^{-1}=[g^{-1}].
\]
These maps are well defined and smooth, and the induced diagram

$$\begin{tikzcd} \mathcal{G} \arrow[r] \arrow[d, shift right] \arrow[d, shift left] & \mathcal{G}/H_1 \arrow[d, shift right] \arrow[d, shift left] \\ M \arrow[r] & M/H_0 \end{tikzcd}
$$
defines a homomorphism of Lie groupoids.
\end{proof}

\begin{Example}[\cite{blohmann2023hamiltonian}]
Let a Lie group \(G\) act on a manifold \(M\), and let \(H\subseteq G\) be a subgroup. Consider the action groupoid \( G\ltimes M \rightrightarrows M \) and the 2-action of the Lie 2-group \( H\times H \rightrightarrows H \) on \( G\ltimes M \rightrightarrows M \) given by
\[
(h_1,h_2)\cdot (g,m)=(h_1 g h_2^{-1},\,h_2\cdot m).
\]
Then \( (G\ltimes M)/(H\times H)\rightrightarrows M/H \) is a Lie groupoid.
\end{Example}

\subsection{Reduction of symplectic groupoids}
\begin{Definition}
Let \( (\mathcal{G},\omega)\rightrightarrows (M,\pi) \) be a symplectic groupoid.
A left $2$-action of the Lie $2$-group
\( H\times H \rightrightarrows H \)
on \( \mathcal{G}\rightrightarrows M \) is called \emph{Hamiltonian} if
\(H\times H\) acts on \( (\mathcal{G},\omega) \) with moment map
\( \mu_{\mathcal{G}} \), and \(H\) acts on \( (M,\pi) \) by Poisson diffeomorphisms
with moment map \( \mu_M \), such that the induced map
\begin{equation}\label{4.1}
(\mu_{\mathcal G}, \mu_M)\colon
\mathcal G \rightrightarrows M
\longrightarrow
\text{Lie}(H)^*\times \text{Lie}(H)^* \rightrightarrows \text{Lie}(H)^*
\end{equation}
is a morphism of Lie groupoids, i.e.\ it is compatible with source, target,
units, multiplication, and inversion.
\end{Definition}

\begin{Proposition}\label{Proposition 3.2}
Suppose a Lie 2-group \(H\times H \rightrightarrows H\) has a Hamiltonian action on a symplectic groupoid \((\mathcal{G}, \omega) \rightrightarrows (M, \pi)\). Let \(x \in \mathrm{Lie}(H)^*\) be a regular value of the moment map \(\mu_M\), and let \((x,x) \in \mathrm{Lie}(H)^*\times \mathrm{Lie}(H)^*\) be a regular value of the moment map \(\mu_{\mathcal{G}}\). Assume the stabilizers of \(x\) and \((x,x)\) are \(H\) and \(H\times H\), respectively. If the actions of \(H\times H\) on \(\mu_{\mathcal{G}}^{-1}(x,x)\) and of \(H\) on \(\mu_M^{-1}(x)\) are both free and proper, then the quotient
\[
(\mathcal{G}//(H\times H), \omega_{\mathrm{red}}) \rightrightarrows (M//H, \pi_{\mathrm{red}})
\]
is a symplectic groupoid.
\end{Proposition}

\begin{proof}
Since \((x,x) \rightrightarrows x\) is a (unit) Lie groupoid and \eqref{4.1} is a groupoid morphism, the level sets \( \mu_{\mathcal{G}}^{-1}(x,x) \rightrightarrows \mu_M^{-1}(x) \) form a Lie subgroupoid. By Proposition~\ref{Proposition 3.1}, the quotient \( \mathcal{G}//(H\times H) \rightrightarrows M//H \) is a Lie groupoid. By Marsden–Weinstein reduction, \( \omega_{\mathrm{red}} \) is a symplectic form on \( \mathcal{G}//(H\times H) \); multiplicativity is inherited from \( \omega \).

There is a canonical Lie algebroid isomorphism
\[
\rho_\omega: A(\mathcal{G}) \xrightarrow{\ \cong\ } T^*M,
\]
characterized by \( \langle \rho_\omega[\xi], v\rangle = \omega(\xi,v) \) for \( [\xi]\in \nu(\mathcal{G},M)=A(\mathcal{G}) \) and \( v\in TM \). Similarly define \( \rho_{\omega_{\mathrm{red}}} \). To see that the base Poisson structure equals \( \pi_{\mathrm{red}} \), it suffices to show
\[
\rho_{\omega_{\mathrm{red}}}: A(\mathcal{G}//(H\times H))\xrightarrow{\ \cong\ } T^*(M//H)\cong T^*M // H
\]
is a Lie algebroid isomorphism. This follows from the commutative diagram
\[
\begin{tikzcd}
A(\mathcal{G}) \arrow[d, "\rho_\omega"'] & A(\mu_{\mathcal{G}}^{-1}(x,x)) \arrow[l, hook'] \arrow[d, "\cong"] \arrow[r] & A(\mathcal{G}//(H\times H)) \arrow[d, "\rho_{\omega_{\mathrm{red}}}"] \\
T^*M & \mu_{T^*M}^{-1}(x) \arrow[l, hook'] \arrow[r] & T^*(M//H)
\end{tikzcd}
\]
where \( \mu_{T^*M} \) is the moment map for the lifted \(H\)-action on \(T^*M\). In particular, \( A(\mathcal{G}//(H\times H)) \) is the reduction of \( A(\mathcal{G}) \), and the claim follows.
\end{proof}

\begin{Example}
Let \( (M,\omega) \) be a symplectic manifold with a Hamiltonian action of a Lie group \( G \) and moment map \( \mu_M:M\to \mathfrak{g}^* \) such that \(G\) acts freely and properly on \( \mu_M^{-1}(0) \). Consider \( \mathcal{G}=M\times M \rightrightarrows M \) with symplectic form \( (\omega,-\omega) \). The action of \( G\times G \) on \( M\times M \) is Hamiltonian with moment map \( \mu_{\mathcal{G}}=(\mu_M,\mu_M) \). The Lie 2-group \( G\times G \rightrightarrows G \) acts Hamiltonianly on \( \mathcal{G} \), and the reduction yields the symplectic groupoid
\[
(M_{\mathrm{red}}\times M_{\mathrm{red}},\,\omega_{\mathrm{red}},-\omega_{\mathrm{red}})\ \rightrightarrows\ (M_{\mathrm{red}},\omega_{\mathrm{red}}),
\]
where \( M_{\mathrm{red}}=\mu_M^{-1}(0)/G \).
\end{Example}
\begin{Example}[Pair groupoid reduction with rotations]
Let $M=\mathbb{R}^2$ with $\omega=dx\wedge dy$, and let $G=S^1$ act by rotations with moment map $\mu_M(x,y)=\tfrac12(x^2+y^2)$. 
Let $\mathcal{G}=M\times M \rightrightarrows M$ be the pair groupoid with multiplicative form $\Omega=pr_1^*\omega-pr_2^*\omega$, and let $G\times G \rightrightarrows G$ act by $(h_1,h_2)\cdot(m_1,m_2)=(h_1\!\cdot m_1,h_2\!\cdot m_2)$. 
Then the 2-action is Hamiltonian with $\mu_{\mathcal{G}}=(\mu_M,\mu_M)$. For $x=r^2/2>0$, freeness/properness hold on $\mu_M^{-1}(x)\cong S^1$, and the reduced symplectic groupoid is the unit groupoid of a point:
\[
(\mathcal{G}//(G\times G),\Omega_{\mathrm{red}})\ \rightrightarrows\ (M//G,\pi_{\mathrm{red}})
\ \cong\ (\{\ast\},0)\ \rightrightarrows\ \{\ast\}.
\]
\end{Example}

\begin{Example}[Action groupoid and conjugation]
Let a Lie group $G$ act on itself by conjugation, and consider $\mathcal{G}=G\ltimes G \rightrightarrows G$. 
The Lie 2-group $G\times G \rightrightarrows G$ acts by $(h_1,h_2)\cdot(g,x)=(h_1 g h_2^{-1},\,h_2 x h_2^{-1})$. 
On open sets where the action is free/proper and for a regular conjugacy class $\mathcal{C}\subset G$, the reduction of $\mathcal{G}$ at $\mathcal{C}$ yields a finite-dimensional symplectic groupoid integrating the induced Poisson structure on a neighborhood of $[\mathcal{C}]$ in $G//G$. 
This provides a group-valued (Cartan–Dirac) counterpart of the pair-groupoid reduction.
\end{Example}
\end{appendix}
\nocite{*}
\bibliographystyle{plain}
\bibliography{Sources}
\end{document}